\documentclass[12pt,a4paper]{amsart}
\usepackage{hyperref}
\usepackage[numbers,sort & compress]{natbib}
\usepackage{graphics,layout,indentfirst,amsmath,amsthm,amssymb}
\newtheorem{theorem}{Theorem}
\newtheorem{remark}[theorem]{Remark}
\newtheorem{corollary}[theorem]{Corollary}
\theoremstyle{theorem}
\newtheorem{example}[theorem]{Example}
\newtheorem{lemma}[theorem]{Lemma}
\theoremstyle{definition} \theoremstyle{definition}
\newtheorem{defn}[theorem]{Definition}

\numberwithin{equation}{section} 
\makeatletter
\@namedef{subjclassname@2010}{2010 Mathematics Subject
Classification} \makeatother
\begin{document}
\title[Applications of third order differential subordination and superordination 
...]{Applications of third order differential subordination and superordination involving generalized Struve function}

\author[P. Gochhayat]{P. Gochhayat}
\address{Department of Mathematics \\
	Sambalpur University\\
	Jyoti Vihar 768019\\
	Burla, Sambalpur, Odisha\\ India}
\email{pgochhayat@gmail.com}
\author[A. Prajapati]{A. Prajapati}
\address{Department of Mathematics \\
	Sambalpur University\\
	Jyoti Vihar 768019\\
	Burla, Sambalpur, Odisha\\ India}
\email{anujaprajapati49@gmail.com}
  
   \date{}
\thanks{The present investigation of the second author is supported
under the   INSPIRE fellowship, Department of Science and
Technology, New Delhi, Government of India, Sanction Letter No.
REL1/2016/2/2015-16.} 
\begin{abstract} 
In the present paper, we derive the  third-order differential subordination and superordination results for some analytic univalent functions defined in the unit disc. These results are associated with generalized Struve functions and are obtained by considering suitable classes of admissible functions. As a consequence, the dual problems which yield the  sandwich type relations are presented.
\end{abstract}


\maketitle{}
\section{Introduction and Preliminaries}\label{sec1}
\subsection{Struve function}The special functions have great importance in  geometric function theory especially after the proof of  famous Bieberbach conjecture which is solved by de-Branges \cite{branges}. Since  there are extensive literature dealing with various geometric properties of certain subclasses of analytic univalent function involving special functions such as  the generalized Gaussian hypergeometric function, Kummer's function and  Bessel functions etc. In the present paper, we are dealing with one of such function which is introduced and studied by Struve \cite{struve} (also see \cite{abramowitz}, \cite{watson}),  is the series solution of inhomogeneous second order Bessel differential equation. Struve functions are applied to various areas of applied mathematics and physics. In \cite{baricz2}, and the reference therein various applications   are illustrated  in different context and also for its application in geometric function theory we refer \cite{AMN},  \cite{BDOY}, \cite{BKP}, \cite{BPS},  \cite{BY}, etc. 

Let $S_{p}$ denote the \textit{Struve function of order $p$ } is of the form
\begin{equation}\label{z3}
S_{p}(z)=\sum_{n=0}^{\infty}\frac{(-1)^{n}}{\Gamma(n+3/2)\Gamma(p+n+3/2)} \left(\frac{z}{2}\right)^{2n+p+1},\qquad   \qquad  (z \in \mathbb{C}),
\end{equation}  
 which is a particular solution of  second-order in-homogeneous Bessel differential equation
 \begin{equation}\label{z2}
 z^{2}w^{\prime\prime}(z)+z w^{\prime}(z)+(z^{2}-p^{2})w(z)=\frac{4(z/2)^{p+1}}{\sqrt{\pi}\Gamma(p+1/2)},
 \end{equation}
 where $p$ is an unrestricted real (or complex) number. Note that the homogeneous part of (\ref{z2}) is the  Bessel equation. 
Consider the following differential equation
\begin{equation}\label{z4}
z^{2}w^{\prime\prime}(z)+z w^{\prime}(z)-(z^{2}+p^{2})w(z)=\frac{4(z/2)^{p+1}}{\sqrt{\pi}\Gamma(p+1/2)},
\end{equation}
which differs from (\ref{z2}) only in the coefficient of $w$. The particular solution of (\ref{z4}) is called the \textit{modified Struve function of order $p$} (\cite{nicholson}, also see\cite{zhang}) and is defined by the formula ,
\begin{equation}\label{z5}
L_{p}(z)=-ie^{-ip \pi/2} S_{p}(iz)=\sum_{n=0}^{\infty} \frac{1}{\Gamma(n+3/2)\Gamma(p+n+3/2)}\left(\frac{z}{2}\right)^{2n+p+1},~  \qquad (z\in \mathbb{C}).
\end{equation} 

Recently, Orhan and Yagmur (\cite{orhan} and \cite{yagmur}, also see \cite{habibullah}) considered  the following   second order inhomogeneous linear differential equation
\begin{equation}\label{z6}
z^{2}w^{\prime\prime}(z)+bz w^{\prime}(z)+(cz^{2}-p^{2}+(1-b)p)w(z)=\frac{4(z/2)^{p+1}}{\sqrt{\pi}\Gamma(p+b/2)},~ (b,c,p \in \mathbb{C}).
\end{equation}
If we choose $b=1$ and $c=1,$ then we get (\ref{z2}), and if we choose $b=1$ and $c=-1,$ then we get (\ref{z4}). This implies that (\ref{z6}) generalizes (\ref{z2}) and (\ref{z4}). Moreover, this permits to study the Struve function and modified Struve function. A particular solution of the differential equation (\ref{z6}), which is denoted by $W_{p,b,c}(z),$ is called the \textit{generalized Struve function  of order $p$}. We have  the following series representation of the function $W_{p,b,c}(z)$, 
\begin{equation}\label{z7}
W_{p,b,c}(z)=\sum_{n=0}^{\infty}\frac{(-c)^{n}}{\Gamma(n+3/2)\Gamma(p+n+(b+2)/2)}\left(z/2\right)^{2n+p+1},\quad ( z\in \mathbb{C}).
\end{equation}
The above series  is convergent in whole complex plane $\mathbb C$, however it is not univalent in the open unit disc $\mathcal{U}:=\{z:z\in\mathbb{C}~\text{and}~|z|< 1\}.$ Rewriting (\ref{z7}) with a suitable transformation on $z$, we have
\begin{equation}\label{z8}
U_{p,b,c}(z):= 2^{p}\sqrt{\pi}\Gamma \left(p+\frac{b+2}{2}\right)z^{\frac{-p-1}{2}} W_{p,b,c}(\sqrt{z}).
\end{equation}
It is pertinent that the equation (\ref{z8}) is analytic in $\mathbb C$ and has following Taylor series expansion:
\begin{equation}\label{z9}
U_{p,b,c}(z)= \sum_{n=0}^{\infty} \frac{(-c/4)^{n}}{(3/2)_{n}(a)_{n}}z^{n},\quad ( z\in \mathbb{C}),
\end{equation} 
where $ a=p+(b+2)/2 \neq 0,-1,-2, \cdots$, and 
 $(\lambda)_n$ is the Pochhammer symbol (or \textit{shifted
factorial}) defined in terms of the gamma function, by
\begin{gather*}
(\lambda)_n=\dfrac{\Gamma(\lambda+n)}{\Gamma(\lambda)}=\begin{cases} 1 & (n=0),\\
\lambda(\lambda+1)(\lambda+2)\cdots (\lambda+n-1)& (n \in \mathbb
N:= \{1,2,\cdots\}).
\end{cases}\end{gather*}
 In a recent work Habibullah et al. (\cite{habibullah}) derived the conditions on parameters $p,~b$ and $c$ such that $zU_{p,b,c}(z)$ is univalent in $\mathcal U$. For simplicity, we write $U_{a,c}(z)=zU_{p,b,c}(z)$. 
 Let $\mathcal{H}$  be  the class  of  functions  analytic  in  $\mathcal {U}
 $.  Denote $\mathcal{H}[\kappa,n]\quad(n \in \mathbb{N}:=\{1,2,3,\cdots\})$, the subclass of $\mathcal{H}$ consists of functions of the form
 $f(z)=\kappa+a_{n}z^{n}+a_{n+1}z^{n+1}+\cdots, \quad(z \in \mathcal{U})$ and $\mathcal{A}(\subset \mathcal{H})$ be the class of  functions analytic in $\mathcal U$  and has the Taylor-Maclaurin  series representation
 \begin{equation*}\label{z1}
 \mathit{f}(z)=z+\sum_{n=2}^{\infty}a_{n}z^{n},\qquad(z \in \mathcal{U}).
 \end{equation*}
We consider a new linear operator $S_{a,c}: \mathcal{A}\longrightarrow \mathcal{A},$ which is defined by the Hadamard product 
\begin{equation}\label{z10}
S_{a,c}f(z)= U_{a,c}(z)* f(z)=z + \sum_{n=1}^{\infty} \frac{(-c/4)^{n}}{(3/2)_{n}(a)_{n}}a_{n+1}z^{n+1},\qquad(z \in \mathcal{U}),
\end{equation}
where $*$ denote the convolution or Hadamard product. Note that convolutions of two analytic functions is also analytic \cite{ruscheweyh12}.
It is easy to verify from the Definition (\ref{z10}) that
\begin{equation}\label{z11}
z(S_{a+1,c}f(z))^{\prime}=a S_{a,c}f(z)-(a-1)S_{a+1,c}f(z),
\end{equation}
where $a=p+(b+1)/2 \neq 0,-1,-2 \cdots.$ The function $S_{a,c}f(z)$ is an elementary transform of the generalized hypergeometric function \cite{andrew} defined by 
\begin{gather*}
_{q}F_{s}(\alpha_{1}, \cdots , \alpha_{q};\beta_{1}, \cdots,\beta_{s}:z)=\sum_{n=0}^{\infty}\frac{(\alpha_{1})_{n}\cdots (\alpha_{q})_{n}}{(\beta_{1})_{n} \cdots (\beta_{s})_{n}}\frac{z^{n}}{n!},\qquad (z\in\mathcal{U}),\\ (\alpha_{i}\in \mathbb{C}; \beta_{j}\in \mathbb{C}\setminus \mathbb{Z}_{0}^{-}; q\leq s+1; q, s \in \mathbb{N}\cup\{0\}; i=1, 2, \cdots, q; j=1,2,\cdots,s).
\end{gather*}
In terms of $_{1}F_{2}$ hypergeometric function, Struve functions defined in (\ref{z10}) is rewritten as: 
\begin{equation*}
S_{a,c}(z)=z _{1}F_{2}\left(1;\frac{3}{2}, a;\frac{-c}{4}z\right)*f(z).
\end{equation*}
We observed that, for suitable choices of the parameters $b$ and $c,$ we obtain some new operators:
\begin{itemize}
\item[(i)] Putting $b=c=1$ in (\ref{z9}), we have the operator $\mathfrak{S}:\mathcal{A}\longrightarrow \mathcal{A}$ familiar with Struve function , defined by
\begin{equation}\label{z12}
\mathfrak{S}f(z)=zU_{p,1,1}(z)*f(z)=z+\sum_{n=1}^{\infty}\frac{(-1)^{n}a_{n+1}z^{n+1}}{(p+3/2)_{n}(2n+1)!}.
\end{equation}
\item[(ii)] Putting $b=1$ and $c=-1$ in (\ref{z9}), we obtain the operator $\mathfrak{I}:\mathcal{A}\longrightarrow \mathcal{A}$ related with modified Struve function, defined by
\begin{equation}\label{z13}
\mathfrak{I}f(z)=zU_{p,1,-1}(z)*f(z)=z+\sum_{n=1}^{\infty} \frac{a_{n+1}z^{n+1}}{(p+3/2)_{n}(2n+1)!}.
\end{equation}
\end{itemize}
Miller and Mocanu \cite{miller4} investigated the dual problem of differential subordination, whereas Bulboaca \cite{bulboaca} investigated both subordination and superordination related results. The theory of first and second order differential subordination and superordination have been used by numerous authors to solve various  problems in this field. For some of the recent works of Gochhayat and others on this direction see \cite{baricz,baricz1,cho,xu,pg,
pg1,pg2,srivastava,cho1,cho2,pg123,mishra25}, also see the monographs \cite{miller3} and \cite{bulboaca} and the reference therein.  
Recently, Antonino and Miller \cite{antonion}  extended the theory of second order differential subordination in the open unit disk $\mathcal{U}$ introduced by Miller and Mocanu \cite{miller3} to the third order case. However, the concept of  third order differential subordination have originally found in the work of Ponnusamy and Juneja \cite{ponnusamy}, ( also see \cite{miller3}). In 2014, Tang et al. \cite{tang} introduced the concept of third order differential superordination, which is a generalization of the second order differential superordination. 
In the recent years, few works have been carried out on results related to the third order differential subordination and superordination  in the different context. For example see (\cite{farzana,ibrahim,jeyaraman, raducanu,tang,tang1,tang2}). In this present investigation our aim is to determine third order differential subordination and superordination  of generalized Struve function by using the technique developed in \cite{antonion} and \cite{tang}. 

\subsection{Basic facts on differential subordination}In order to achieve our aim in this section, we recall some definitions and preliminary results from the theory of differential subordination and superordination. 

Suppose that $f$ and $g$ are in $\mathcal H$. We say that $f$ is \textit{subordinate} to $g$, (or $g$ is
\textit{superordinate} to $f$), write as $f \prec g ~\text{in}~\mathcal U~\text{ or}~ f(z) \prec g(z) \quad (z \in \mathcal U),$ if there exists a function $\omega\in \mathcal H$, satisfying the conditions of the Schwarz lemma
$(~\text{i.e.}~\omega(0)=0$ and $|\omega(z)|< 1)$ such that
$f(z)=g(\omega(z))\quad(z \in \mathcal U).$ It follows that
$
f(z) \prec g(z)\;(z \in \mathcal U) \Longrightarrow f(0)=g(0) ~
\text{and} ~ f(\mathcal U) \subset g(\mathcal U).
$ In particular, if $g$ is \textit{univalent} in $\mathcal U$, then
the reverse implication also holds (cf.\cite{miller3}). 
\begin{defn}\label{d3a}[\cite{antonion}, Definition 1, p.440].
	Let $\psi:\mathbb{C}^{4} \times \mathcal{U}\longrightarrow \mathbb{C}$ and the function $h(z)$ be univalent in $\mathcal{U}.$ If the function $p(z)$ is analytic in $\mathcal{U}$ and satisfies the following third-order differential subordination
	\begin{equation}\label{p1}
	\psi(p(z),zp^{\prime}(z),z^{2}p^{\prime\prime}(z),z^{3}
	p^{\prime\prime\prime}(z);z)\prec h(z),
	\end{equation}
	then $p(z)$ is called a \emph{solution} of the differential subordination.
\end{defn}
A univalent function $q(z)$ is called a \emph{dominant} of the solutions of the differential subordination,or, more simply, a dominant if $p(z)\prec q(z)$ for all $p(z)$ satisfying (\ref{p1}). A dominant $\tilde{q}(z)$ 
that satisfies $\tilde{q}(z)\prec q(z)$ 
for all dominants $q(z)$ of (\ref{p1}) is said to be the \emph{best dominant}. 
\begin{defn}\label{d5}\cite{tang}
Let $\psi:\mathbb{C}^{4} \times \mathcal{U}\longrightarrow \mathbb{C}$ and the function $h(z)$ be univalent in $\mathcal{U}.$ If the function $p(z)$ and
\begin{equation*}
\psi(p(z),zp^{\prime}(z),z^{2}p^{\prime\prime}(z),z^{3}
p^{\prime\prime\prime}(z);z)
\end{equation*}
	are univalent in $\mathcal{U}$ and satisfies the following third-order differential superordination
	\begin{equation}\label{p2}
	h(z) \prec \psi(p(z),zp^{\prime}(z),z^{2}p^{\prime\prime}(z),z^{3}
	p^{\prime\prime\prime}(z);z),
	\end{equation}
	then $p(z)$ is called a \emph{solution} of the differential superordination. An analytic function $q(z)$ is called a \emph{subordinant} of the solutions of the differential superordination, or more simply a subordinant, if $ q(z) \prec p(z)$ for all $p(z)$ satisfying (\ref{p2}).
\end{defn}
A univalent subordinant $\tilde{q}(z)$ that satisfies $q(z)\prec \tilde{q}(z)$ 
for all subordinant $q(z)$ of (\ref{p2}) is said to be the \emph{best subordinant}. We note that both the best dominant and best subordinant are unique up to rotation of $\mathcal U$.
\begin{defn}\label{d4}[\cite{antonion}, Definition 2, p.441].
	Let $\mathcal{Q}$ denote the set of functions $q$ that are analytic and univalent on the set $\overline{\mathcal{U}}\setminus E(q),$ where $E(q)=\{\xi:\xi \in \partial \mathcal{U}:\lim_{z\rightarrow \xi}q(z)= \infty\},$ and are such that $\min \mid q^{\prime}(\xi) \mid=\rho >0$ for $\xi \in \partial \mathcal{U} \setminus E(q).$
Further, let the subclass of $\mathcal{Q}$ for which $q(0)=\kappa$ be denoted by $\mathcal{Q}(\kappa),\mathcal{Q}(0)=\mathcal{Q}_{0}$ and $\mathcal{Q}(1)=\mathcal{Q}_{1}.$
\end{defn}
The subordination methodology is applied to an appropriate class of admissible functions. The following class of admissible functions was given by Antonino and Miller.
\begin{defn}\label{d2a}[\cite{antonion}, Definition 3, p.449].
	Let $\Omega$ be a set in $\mathbb{C}$ and $q \in \mathcal{Q}$  and $n \in \mathbb{N} \setminus \{1\}.$ The class of admissible functions $\Psi_{n}[\Omega,q]$ consists of those functions $\psi:\mathbb{C}^{4} \times \mathcal{U}\longrightarrow \mathbb{C}$ achieving the following admissibility conditions:
	$$\psi(r,s,t,u;z) \not\in \Omega$$
	whenever
	$$r=q(\zeta),s=k\zeta q^{\prime}(\zeta), \Re\left(\frac{t}{s}+1\right)\geq k \Re\left(\frac{\zeta q^{\prime\prime}(\zeta)}{q^{\prime}(\zeta)}+1\right),$$
	and $$\Re\left(\frac{u}{s}\right)\geq k^{2} \Re\left(\frac{\zeta^{2} q^{\prime\prime\prime}(\zeta)}{q^{\prime}(\zeta)}\right),$$
	where $z \in \mathcal{U},\zeta \in \partial \mathcal{U} \setminus E(q),$ and $k\geq n.$
\end{defn}
The next lemma is the foundation result in the theory of third-order differential subordination.
\begin{lemma}\label{t4}[\cite{antonion}, Theorem 1, p.449].
Let $p \in \mathcal{H}[\kappa,n]$ with $n\geq 2,$ and $q \in \mathcal{Q}(\kappa)$ achieving the following conditions:
\begin{equation*}
\Re\left(\frac{\zeta q^{\prime\prime}(\zeta)}{q^{\prime}(\zeta)}\right)\geq 0,\qquad \left|\frac{zq^{\prime}(z)}{q^{\prime}(\zeta)}\right|\leq k,
\end{equation*}	
 where $z \in \mathcal{U},\zeta \in \partial \mathcal{U} \setminus E(q),$ and $k\geq n.$ If $\Omega$ is a set in $\mathbb{C},\psi \in \Psi_{n}[\Omega,q]$ and 
\begin{equation*}	
\psi \left(p(z),zp^{\prime}(z),z^{2}p^{\prime\prime}(z),z^{3}p^{\prime\prime\prime}(z);z\right) \subset \Omega,
\end{equation*}	
then 
\begin{equation*}
p(z) \prec q(z)\qquad(z \in \mathcal{U}).
\end{equation*}
\end{lemma}
\begin{defn}\label{d6}\cite{tang}
Let $\Omega$ be a set in $\mathbb{C}$ and $q \in \mathcal{H}[\kappa,n]$  and $ q^{\prime}(z)\neq 0.$ The class of admissible functions $\Psi_{n}^{\prime}[\Omega,q]$ consists of those functions $\psi:\mathbb{C}^{4} \times \mathcal{\overline{U}}\longrightarrow \mathbb{C}$ that satisfy the following admissibility conditions:
\begin{equation*}
\psi(r,s,t,u;z) \in \Omega
\end{equation*}
whenever
\begin{equation*}
r=q(z),s=\frac{zq^{\prime}(z)}{m}, \Re\left(\frac{t}{s}+1\right)\leq \frac{1}{m} \Re\left(\frac{z q^{\prime\prime}(z)}{q^{\prime}(z)}+1\right),
\end{equation*}
and 
\begin{equation*}
\Re\left(\frac{u}{s}\right)\leq \frac{1}{m^{2}} \Re\left(\frac{z^{2} q^{\prime\prime\prime}(z)}{q^{\prime}(z)}\right),
\end{equation*}
where $z \in \mathcal{U},z \in \partial \mathcal{U}$ and $m\geq n \geq 2.$
\end{defn} 
\begin{lemma}\label{t5}\cite{tang}
Let $p \in \mathcal{H}[\kappa,n]$ with $\psi \in \Psi_{n}^{\prime}[\Omega,q].$ If
\begin{equation*}
\psi(p(z),zp^{\prime}(z),z^{2}p^{\prime\prime}(z),z^{3}p^{\prime\prime\prime}(z);z)
\end{equation*}	
is univalent in $\mathcal{U}$ and $p \in \mathcal{Q}(\kappa)$ satisfying the following conditions:
\begin{equation*}
	\Re\left(\frac{z q^{\prime\prime}(z)}{q^{\prime}(z)}\right)\geq 0,\qquad \left|\frac{zp^{\prime}(z)}{q^{\prime}(z)}\right|\leq m, 
\end{equation*}	
	where $z \in \mathcal{U},\zeta \in \partial \mathcal{U},$ and $m\geq n \geq 2,$ then
\begin{equation*}	
\Omega \subset \{\psi \left(p(z),zp^{\prime}(z),z^{2}p^{\prime\prime}(z),z^{3}
	p^{\prime\prime\prime}(z);z\right):z \in \mathcal{U}\}
\end{equation*}
implies that 
\begin{equation*}	
q(z) \prec p(z)\qquad  (z \in \mathcal{U}).
\end{equation*}
\end{lemma}
Recently, by making use of Hohlov operator Gochhayat et al. (cf.\cite{mishra}) have derived various third order differential subordination and superordination related results.  In this  investigation, by considering suitable classes of admissible functions, we obtained some interesting inclusion results on third order differential subordination and superordination involving $S_{a,c}$. More precisely, we have shown that the sandwich-type relations of the form $$q_1(z)\prec \Xi(z) \prec
q_2(z), \qquad (z\in \mathcal U),$$
 holds, where $q_1,~q_2$ are univalent in $\mathcal U$ with suitable normalizations, and $\Xi(z)$ is one of the variant of ${S_{a,c}  f(z)}$. 
\section{Results based on differential  subordination}
In this section the following class of admissible functions is introduced which are required to prove the main third-order differential subordination theorems involving the operator $S_{a,c}$  defined by (\ref{z10}).
\begin{defn}\label{d1}
Let $\Omega$ be a set in $\mathbb{C}$ and $q \in \mathcal{Q}_{0}\bigcap \mathcal{H}_{0}$. The class of admissible function $\Phi_{\mathit{S}}[\Omega,q]$ consists of those functions $\phi:\mathbb{C}^{4}\times \mathcal{U}\longrightarrow\mathbb{C}$ that satisfy the following admissibility conditions:
\begin{equation*}
\phi(\alpha,\beta,\gamma,\delta;z) \not\in \Omega
\end{equation*}
whenever
\begin{equation*}
\alpha=q(\zeta),\beta=\frac{k\zeta q^{\prime}(\zeta)+(a-1)q(\zeta)}{a},
\end{equation*}
\begin{equation*}
\Re \left(\frac{a(a-1)\gamma-(a-2)(a-1)\alpha}{a\beta-(a-1)\alpha}-(2a-3)\right)
\geq k \Re\left(\frac{\zeta q^{\prime\prime}(\zeta)}{q^{\prime}(\zeta)}+1\right),
\end{equation*}
and 
\begin{equation*}
\Re
\left(\frac{a(a-1)((1-a)\alpha+(3a\beta+(1-3a)\gamma+(a-2)\delta)}{\alpha+a(\beta-\alpha)}\right)
\geq k^{2} \Re\left(\frac{\zeta^{2} q^{\prime\prime\prime}(\zeta)}{q^{\prime}(\zeta)}\right),
\end{equation*}
where $z \in \mathcal{U},\zeta \in \partial \mathcal{U} \setminus E(q),$ and $k\geq 2.$
\end{defn}
\begin{theorem}\label{t6}
Let $\phi \in \Phi_{S}[\Omega,q]$. If the function $f \in \mathcal{A}$ and $q \in \mathcal{Q}_{0}$ satisfy the following conditions:
\begin{equation}\label{z16}
\Re\left(\frac{\zeta q^{\prime\prime}(\zeta)}{q^{\prime}(\zeta)}\right)\geq 0 \quad,\left|\frac{S_{a,c}f(z)}{q^{\prime}(\zeta)}\right|\leq k,
\end{equation}
and 
\begin{equation}\label{z17}
\left\{\phi(S_{a+1,c}f(z),S_{a,c}f(z),S_{a-1,c}f(z),S_{a-2,c}f(z);z):z \in \mathcal{U}\right\}\subset \Omega,
\end{equation}
then 
\begin{equation*}
S_{a+1,c}f(z) \prec q(z)\qquad(z \in \mathcal{U}).
\end{equation*}
\begin{proof}
Define the analytic function $p(z)$ in $\mathcal{U}$ by
\begin{equation}\label{z18}
p(z)=S_{a+1,c}f(z).
\end{equation}
From equation (\ref{z18}) and (\ref{z11}), we have
\begin{equation}\label{z19}
S_{a,c}f(z)=\frac{zp^{\prime}(z)+(a-1)p(z)}{a}.
\end{equation}
By similar argument yields,
\begin{equation}\label{z20}
S_{a-1,c}f(z)=\frac{z^{2}p^{\prime\prime}(z)+2z(a-1)p^{\prime}(z)+(a-2)(a-1)p(z)}{a(a-1)}
\end{equation}
and
\begin{equation}\label{z21}
S_{a-2,c}f(z)=\frac{z^{3}p^{\prime\prime\prime}(z)+3(a-1)z^{2}p^{\prime\prime}(z)+3(a-1)(a-2)zp^{\prime}(z)+(a-1)(a-2)(a-3)p(z)}{a(a-1)(a-2)}.
\end{equation}
Define the transformation from $\mathbb{C}^{4}$ to $\mathbb{C}$ by
\begin{equation*}
\alpha(r,s,t,u)=r,\qquad \beta(r,s,t,u)=\frac{s+(a-1)r}{a},
\end{equation*}
\begin{equation}\label{z22}
\gamma(r,s,t,u)=\frac{t+2(a-1)s+(a-2)(a-1)r}{a(a-1)}
\end{equation}
and
\begin{equation}\label{z23}
\delta(r,s,t,u)=\frac{u+3(a-1)t+3(a-1)(a-2)s+(a-1)(a-2)(a-3)r}{a(a-1)(a-2)}.
\end{equation}
Let
\begin{multline}\label{z24}
\psi(r,s,t,u)=\phi(\alpha,\beta,\gamma,\delta;z)=
\phi \bigg(r,\frac{s+(a-1)r}{a},\frac{t+2(a-1)s+(a-2)(a-1)r}{a(a-1)},\\ \frac{u+3(a-1)t+3(a-1)(a-2)s+(a-1)(a-2)(a-3)r}{a(a-1)(a-2)};z\bigg).
\end{multline}
The proof will make use of  Lemma \ref{t4}. Using equations (\ref{z18}) to (\ref{z21}), and from (\ref{z24}), we have
\begin{equation}\label{z25}
\psi \left(p(z),zp^{\prime}(z),z^{2}p^{\prime\prime}(z),z^{3}p^{\prime\prime\prime}(z);z\right)=\phi \left(S_{a+1,c}f(z),S_{a,c}f(z),S_{a-1,c}f(z),S_{a-2,c}f(z);z \right).
\end{equation} 
Hence,(\ref{z17}) becomes
\begin{equation*}
\psi \left(p(z),zp^{\prime}(z),z^{2}p^{\prime\prime}(z),z^{3}p^{\prime\prime\prime}(z);z\right) \in \Omega.
\end{equation*}
Note that
\begin{equation*}
\frac{t}{s}+1=\frac{a(a-1)\gamma-(a-2)(a-1)\alpha}{a\beta-(a-1)\alpha}-(2a-3),
\end{equation*}
and
\begin{equation*}
\frac{u}{s}=\frac{a(a-1)((1-k)\alpha+3a\beta+(1-3a)\gamma+(a-2)\delta)}{\alpha+a(\beta-\alpha)}.
\end{equation*}
Thus, the admissibility condition for $\phi \in \Phi_{\mathcal{I}}[\Omega,q]$ in Definition \ref{d1} is equivalent to the admissibility condition for $\psi \in \Psi_{2}[\Omega,q]$ as given in Definition \ref{d2a} with $n=2$. Therefore, by using (\ref{z16}) and Lemma \ref{t4}, we have
\begin{equation*}
S_{a+1,c}f(z)\prec q(z).
\end{equation*}
This completes the proof of theorem.
\end{proof}
\end{theorem}
The next result is an extension of theorem  \ref{t6} to the case where the behavior of $q(z)$ on $\partial \mathcal{U}$ is not known.
\begin{corollary}\label{c1}
Let $\Omega \subset \mathbb{C}$ and let the function $q$ be univalent in $\mathcal{U}$ with $q(0)=0.$ Let $\phi \in \Phi_{S}[\Omega,q_{\rho}]$ for some $\rho \in (0,1),$ where $q_{\rho}(z)=q(\rho z).$ If the function $f \in \mathcal{A}$ and $q_{\rho}$ satisfy the following conditions
\begin{equation*}
\Re\left(\frac{\zeta q_{\rho}^{\prime\prime}(\zeta)}{q_{\rho}^{\prime}(\zeta)}\right)\geq 0,\qquad\left|\frac{S_{a,c}f(z)}{q_{\rho}^{\prime}(\zeta)}\right|\leq k\quad(z \in \mathcal{U},\zeta \in \partial \mathcal{U}\setminus E(q_{\rho})),
\end{equation*}
and
\begin{equation*}
\phi\left(S_{a+1,c}f(z),S_{a,c}f(z),S_{a-1,c}f(z),S_{a-2,c}f(z);z\right) \in \Omega,
\end{equation*}
then
\begin{equation*}
S_{a+1,c}f(z)\prec q(z)\qquad(z \in \mathcal{U}).
\end{equation*}
\begin{proof}
From Theorem \ref{t6}, then $S_{a+1,c}f(z)\prec q_{\rho}(z)$. The result asserted by Corollary \ref{c1} is now deduced from the following subordination property $q_{\rho}(z)\prec q(z)\qquad(z \in \mathcal{U}).$ 
\end{proof}
\end{corollary}
If $\Omega \neq \mathbb{C}$ is a simply connected domain, then $\Omega=h(\mathcal{U})$ for some conformal mapping $h(z)$ of $\mathcal{U}$ onto $\Omega.$ In this case, the class $\Phi_{S}[h(\mathcal{U}),q]$ is written  as $\Phi_{S}[h,q].$ The following result follows immediately as a consequence of Theorem  \ref{t6}.
\begin{theorem}\label{t7}
Let $\phi \in \Phi_{S}[h,q]$. If the function $f \in \mathcal{A}$ and $q \in \mathcal{Q}_{0}$ satisfy the following conditions:
\begin{equation}\label{z26}
\Re\left(\frac{\zeta q^{\prime\prime}(\zeta)}{q^{\prime}(\zeta)}\right)\geq 0, \quad \left|\frac{S_{a,c}f(z)}{q^{\prime}(\zeta)}\right|\leq k,
\end{equation}
and 
\begin{equation}\label{z27}
\phi(S_{a+1,c}f(z),S_{a,c}f(z),S_{a-1,c}f(z),S_{a-2,c}f(z);z)\prec h(z),
\end{equation}
then 
\begin{equation*}
S_{a+1,c}f(z) \prec q(z)\quad(z \in \mathcal{U}).
\end{equation*}
\end{theorem}
The next result is an immediate consequence of Corollary \ref{c1}.
\begin{corollary}\label{c2}
Let $\Omega \subset \mathbb{C}$ and let the function $q$ be univalent in $\mathcal{U}$ with $q(0)=0.$ Let $\phi \in \Phi_{S}[h,q_{\rho}]$ for some $\rho \in (0,1),$ where $q_{\rho}(z)=q(\rho z).$ If the function $f \in \mathcal{A}$ and $q_{\rho}$ satisfy the following conditions
\begin{equation*}
\Re\left(\frac{\zeta q_{\rho}^{\prime\prime}(\zeta)}{q_{\rho}^{\prime}(\zeta)}\right)\geq 0,\qquad \left|\frac{S_{a,c}f(z)}{q_{\rho}^{\prime}(\zeta)}\right|\leq k\qquad(z \in \mathcal{U},\zeta \in \partial \mathcal{U}\setminus E(q_{\rho})),
\end{equation*}
and
\begin{equation*}
\phi(S_{a+1,c}f(z),S_{a,c}f(z),S_{a-1,c}f(z),S_{a-2,c}f(z);z) \prec h(z),
\end{equation*}
then
\begin{equation*}
S_{a+1,c}f(z)\prec q(z)\qquad(z \in \mathcal{U}).
\end{equation*}
\end{corollary}
The following result  yields the best dominant of the differential subordination (\ref{z27}).
\begin{theorem}\label{t8}
Let the function $h$ be univalent in $\mathcal{U}$ and let $\phi : \mathbb{C}^{4}\times \mathcal{U}\longrightarrow \mathbb{C}$ and $\psi$ be given by (\ref{z24}). Suppose that the differential equation
\begin{equation}\label{z28}
\psi(q(z),zq^{\prime}(z),z^{2}q^{\prime\prime}(z),z^{3}
q^{\prime\prime\prime}(z);z)=h(z),
\end{equation}
has a solution $q(z)$ with $q(0)=0,$ which satisfy condition (\ref{z16}). If the function $f \in \mathcal{A}$ satisfies condition (\ref{z27}) and 
\begin{equation*}
\phi(S_{a+1,c}f(z),S_{a,c}f(z),S_{a-1,c}f(z),S_{a-2,c}f(z);z)
\end{equation*}
is analytic in $\mathcal{U},$ then
\begin{equation*}
S_{a+1,c}f(z)\prec q(z)
\end{equation*}
and $q(z)$ is the best dominant.
\begin{proof}
From Theorem \ref{t6}, we have $q$ is a dominant of (\ref{z27}). Since $q$ satisfies (\ref{z28}), it is also a solution of (\ref{z27}) and therefore $q$ will be dominated by all dominants. Hence $q$ is the best dominant.
This completes the proof of theorem.
\end{proof}
\end{theorem}
In view of Definition \ref{d1}, and in the special case $q(z)=Mz,~~ M>0,$ the class of admissible functions $\Phi_{S}[\Omega,q],$ denoted by  $\Phi_{S}[\Omega,M],$ is expressed as follows.
\begin{defn}\label{d15}
Let $\Omega$ be a set in $\mathbb{C}$  and $M>0$. The class of admissible function $\Phi_{S}[\Omega,M]$ consists of those functions $\phi:\mathbb{C}^{4}\times \mathcal{U}\longrightarrow\mathbb{C}$ such that
\begin{multline}\label{z29}
\phi \bigg(Me^{i\theta},\frac{(k+a-1)Me^{i\theta}}{a},\frac{L+[(2k+a-2)(a-1)]Me^{i\theta}}{a(a-1)},\\ \frac{N+3(a-1)L+[(a-1)(a-2)(3k+a-3)]Me^{i\theta}}{a(a-1)(a-2)};z\bigg) \not\in \Omega,
\end{multline}
where $z \in \mathcal{U},\Re(Le^{-i\theta})\geq (k-1)kM,$ and $\Re(Ne^{-i\theta})\geq 0$ for all $\theta \in \mathbb{R}$ and $k \geq 2$. 
\end{defn}
\begin{corollary}\label{c3}
Let $\phi \in \Phi_{S}[\Omega,M].$ If the function $f \in \mathcal{A}$ satisfies 
\begin{equation*}
\left|S_{a,c}f(z)\right|\leq kM\qquad(z \in \mathcal{U},k\geq 2; M>0)
\end{equation*}
and 
\begin{equation*}
\phi(S_{a+1,c}f(z),S_{a,c}f(z),S_{a-1,c}f(z),S_{a-2,c}f(z);z)\in \Omega,
\end{equation*}
then
\begin{equation*}
\left|S_{a+1,c}f(z)\right|<M.
\end{equation*}
\end{corollary}
 In this special case $\Omega=q(\mathcal{U})=\{w:|w|<M\},$ the class $\Phi_{S}[\Omega,M]$ is simply denoted by $\Phi_{S}[M]$. Corollary \ref{c3} can now be written in the following form:
 \begin{corollary}\label{c4}
 Let $\phi \in \Phi_{S}[M].$ If the function $f \in \mathcal{A}$ satisfies 
\begin{equation*}
\left|S_{a,c}f(z)\right|\leq kM\qquad(z \in \mathcal{U},k\geq 2;M>0),
\end{equation*}
and 
\begin{equation*}
\left|\phi(S_{a+1,c}f(z),S_{a,c}f(z),S_{a-1,c}f(z),S_{a-2,c}f(z);z \right|<M,
\end{equation*}
then
\begin{equation*}
\left|S_{a+1,c}f(z)\right|< M.
\end{equation*}
\end{corollary}
\begin{corollary}\label{c15}
Let $\Re(a)\geq \frac{1-k}{2},\quad k \geq 2$ and $M>0.$ If $f \in \mathcal{A}$
 satisfies 
 $$|S_{a,c}f(z)| \leq M,$$
 then
 $$|S_{a+1,c}f(z)|< M.$$
 \begin{proof}
 This follows from Corollary \ref{c4}  by taking $\phi(\alpha,\beta,\gamma,\delta;z)=\beta=\frac{k+a-1}{a}Me^{i\theta}.$ 
\end{proof}
\end{corollary}
\begin{remark}
For $f(z)=\frac{z}{1-z}$ in Corollary \ref{c15}, we have 
\begin{equation*}
|U_{a,c}(z)|<M\Longrightarrow |U_{a+1,c}(z)|<M,
\end{equation*}
which is a generalization of result given by Prajapat \cite{prajapat}.
\end{remark}
\begin{corollary}\label{c16}
Let $0 \neq a \in \mathbb{C},k\geq 2$ and $M>0.$ If $f \in \mathcal{A}$ satisfies 
$$|S_{a,c}f(z)|\leq kM,$$
and
$$|S_{a,c}f(z)-S_{a+1,c}f(z)|<\frac{M}{|a|},$$
then
$$|S_{a+1,c}f(z)|<M.$$
\begin{proof}
Let $\phi(\alpha,\beta,\gamma,\delta;z)=\beta-\alpha $ and $\Omega=h(\mathcal{U}),$ where $h(z)=\frac{Mz}{a},~~M>0.$ In order to use Corollary \ref{c3}, we need to show that $\phi \in \Phi_{S}[\Omega,M],$ that is, the admissibility condition (\ref{z29}) is satisfied.  This follows since
\begin{equation*}
|\phi(\alpha,\beta,\gamma,\delta;z)|=\left|\frac{(k-1)Me^{i\theta}}{a}\right|\geq \frac{M}{|a|},
\end{equation*}
whenever $z \in \mathcal{U}, \theta \in \mathbb{R}$ and $k \geq 2.$ The required result now follows from Corollary \ref{c3}.

Theorem \ref{t8}  shows that the result is sharp. The differential equation $z q^{\prime}(z)=Mz$ has a univalent solution $q(z)=Mz.$ It follows from Theorem \ref{t8}  that $q(z)=Mz$ is the best dominant.
\end{proof}
\end{corollary}
\begin{example} For $p=\pm 1/2, ~b=1$ and $c=-1$, we have 
$
U_{2,-1}(z)=zU_{1/2,1,-1}(z)=2(\cosh(\sqrt{z})-1)$ \quad and \quad $U_{1,-1}(z)=zU_{-1/2,1,-1}(z)=\sqrt{z}\sinh\sqrt{z},
$
where $U_{p,b,c}$ is given by (\ref{z9}). Further more, taking $f(z)=\frac{z}{1-z}$ in Corollary \ref{c16}, we get
\begin{equation*}
\left|\sqrt{z}\sinh\sqrt{z}-2(\cosh(\sqrt{z})-1)\right|<M\Rightarrow \left|(\cosh(\sqrt{z})-1)\right|<\frac{M}{2}.
\end{equation*}
\end{example}
\begin{defn}\label{d3}
Let $\Omega$ be a set in $\mathbb{C}, q \in \mathcal{Q}_{0} \cap \mathcal{H}_{0}.$ The class of admissible functions $\Phi_{S,1}[\Omega,q]$ consists of those function $\phi:\mathbb{C}^{4}\times \mathcal{U}\longrightarrow \mathbb{C}$ that satisfy the following admissibility condition
\begin{equation*}
 \phi(\alpha,\beta,\gamma,\delta;z) \not\in \Omega
 \end{equation*}
whenever
\begin{equation*}
\alpha=q(\zeta),\beta=\frac{k\zeta q^{\prime}(\zeta)+aq(\zeta)}{a}, 
\end{equation*}
\begin{equation*}
\Re \left(\frac{(a-1)(\gamma-\alpha)}{\beta-\alpha}+(1-2a)
\right)\geq k \Re\left(\frac{\zeta q^{\prime\prime}(\zeta)}{q^{\prime}(\zeta)}+1\right),
\end{equation*}
and 
\begin{eqnarray*}
&&\Re \left(\frac{(a-1)(a-2)(\delta-\alpha)-3a(a-1)(\gamma-2\alpha+\beta)}{\beta-\alpha}+6a^{2}\right)\geq k^{2} \Re\left(\frac{\zeta^{2} q^{\prime\prime\prime}(\zeta)}{q^{\prime}(\zeta)}\right),
\end{eqnarray*}
where $z \in \mathcal{U},\zeta \in \partial \mathcal{U} \setminus E(q),$ and $k\geq 2.$
\end{defn}
\begin{theorem}\label{t9}
Let $\phi \in \Phi_{S,1}[\Omega,q].$ If the function $f \in \mathcal{A}$ and $q \in \mathcal{Q}_{0}$ satisfy the following conditions:
\begin{equation}\label{z30}
\Re\left(\frac{\zeta q^{\prime\prime}(\zeta)}{q^{\prime}(\zeta)}\right)\geq 0,\quad \left|\frac{S_{a,c}f(z)}{zq^{\prime}(\zeta)}\right|\leq k,
\end{equation}
and 
\begin{equation}\label{z31}
\left\{\phi \left(\frac{{S}_{a+1,c}f(z)}{z},\frac{S_{a,c}f(z)}{z},\frac{S_{a-1,c}f(z)}{z},\frac{S_{a-2,c}f(z)}{z};z \right):z \in \mathcal{U}\right\}\subset \Omega,
\end{equation}
then 
\begin{equation*}
\frac{S_{a+1,c}f(z)}{z} \prec q(z)\qquad(z \in \mathcal{U}).
\end{equation*}
\begin{proof}
Define the analytic function $p(z)$ in $\mathcal{U}$ by
\begin{equation}\label{z32}
p(z)=\frac{S_{a+1,c}f(z)}{z}.
\end{equation}
From equation (\ref{z11}) and (\ref{z32}), we have
\begin{equation}\label{z33}
\frac{S_{a,c}f(z)}{z}=\frac{zp^{\prime}(z)+ap(z)}{a}.
\end{equation}
Similar arguments, yields
\begin{equation}\label{z34}
\frac{S_{a-1,c}f(z)}{z}=\frac{z^{2}p^{\prime\prime}(z)+2azp^{\prime}(z)+a(a-1)p(z)}{a(a-1)}
\end{equation}
and
\begin{equation}\label{z35}
\frac{S_{a-2,c}f(z)}{z}=\frac{z^{3}p^{\prime\prime\prime}(z)+3az^{2}p^{\prime\prime}(z)+3a(a-1)zp^{\prime}(z)+a(a-1)(a-2)p(z)}{a(a-1)(a-2)}.
\end{equation}
Define the transformation from $\mathbb{C}^{4}$ to $\mathbb{C}$ by
\begin{equation*}
\alpha(r,s,t,u)=r,\qquad \beta(r,s,t,u)=\frac{s+ar}{a},
\end{equation*}
\begin{equation}\label{z36}
\gamma(r,s,t,u)=\frac{t+2as+a(a-1)r}{a(a-1)},
\end{equation}
and
\begin{equation}\label{z37}
\delta(r,s,t,u)=\frac{u+3at+3a(a-1)s+ra(a-1)(a-2)}{a(a-1)(a-2)}.
\end{equation}
Let
\begin{multline}\label{z38}
\psi(r,s,t,u)=\phi(\alpha,\beta,\gamma,\delta;z)=\phi \bigg(r,\frac{s+ar}{a},\frac{t+2as+a(a-1)r}{a(a-1)},\\ \frac{u+3at+3a(a-1)s+a(a-1)(a-2)r}{a(a-1)(a-2)};z\bigg).
\end{multline}
The proof will make use of  Lemma \ref{t4}. Using equations (\ref{z32}) to (\ref{z35})  and from (\ref{z38}), we have
\begin{equation}\label{z39}
\psi \left(p(z),zp^{\prime}(z),z^{2}p^{\prime\prime}(z),z^{3}p^{\prime\prime\prime}(z);z\right)=\phi \left(\frac{{S}_{a+1,c}f(z)}{z},\frac{S_{a,c}f(z)}{z},\frac{S_{a-1,c}f(z)}{z},\frac{S_{a-2,c}f(z)}{z};z\right). 
\end{equation} 
Hence, (\ref{z31}) becomes
\begin{equation*}
\psi \left(p(z),zp^{\prime}(z),z^{2}p^{\prime\prime}(z),z^{3}p^{\prime\prime\prime}(z);z\right) \in \Omega.
\end{equation*}
Note that
\begin{equation*}
\frac{t}{s}+1=\frac{(a-1)(\gamma-\alpha)+a^{2}v}{\beta-\alpha}+(1-2a)
\end{equation*}
and
\begin{eqnarray*}
\frac{u}{s}=\frac{(a-1)(a-2)(\delta-\alpha)-3a(a-1)(\gamma-2\alpha+\beta)}{\beta-\alpha}+6a^{2}.
\end{eqnarray*}
Thus, the admissibility condition for $\phi \in \Phi_{S,1}[\Omega,q]$ in Definition \ref{d3} is equivalent to the admissibility condition for $\psi \in \Psi_{2}[\Omega,q]$ as given in Definition  \ref{d2a} with $n=2.$ Therefore, by using (\ref{z30}) and Lemma \ref{t4}, we have
\begin{equation*}
\frac{S_{a+1,c}f(z)}{z}\prec q(z).
\end{equation*}
This completes the proof of theorem.
\end{proof}
\end{theorem}
If $\Omega \neq \mathbb{C}$ is a simply connected domain, then $\Omega=h(\mathcal{U})$ for some conformal mapping $h(z)$ of $\mathcal{U}$ onto $\Omega.$ In this case, the class $\Phi_{S,1}[h(\mathcal{U}),q]$ is written  as $\Phi_{S,1}[h,q]$. 
proceeding similarly as in the previous theorem, the following result is an immediate  consequence of Theorem   \ref{t9}.
\begin{theorem}\label{t10}
Let $\phi \in \Phi_{S,1}[h,q].$If the function $f \in \mathcal{A}$ and $q \in \mathcal{Q}_{0}$ satisfy the following conditions:
\begin{equation}\label{z40}
\Re\left(\frac{\zeta q^{\prime\prime}(\zeta)}{q^{\prime}(\zeta)}\right)\geq 0,\left|\frac{S_{a,c}f(z)}{zq^{\prime}(\zeta)}\right|\leq k,
\end{equation}
and 
\begin{equation}\label{z41}
\phi \left(\frac{S_{a+1,c}f(z)}{z},\frac{S_{a,c}f(z)}{z},\frac{S_{a-1,c}f(z)}{z},\frac{S_{a-2,c}f(z)}{z};z \right)\prec h(z),
\end{equation}
then 
\begin{equation*}
\frac{S_{a+1,c}f(z)}{z} \prec q(z)\qquad(z \in \mathcal{U}).
\end{equation*}
\end{theorem}
 In the particular case $q(z)=1+Mz,~~M>0,$ and in view of Definition \ref{d3} the class of admissible functions $\Phi_{S,1}[\Omega,q],$ denoted by  $\Phi_{S,1}[\Omega,M],$ is expressed as follows.
\begin{defn}\label{d4}
Let $\Omega$ be a set in $\mathbb{C}, a \in \mathbb{C} \setminus \{0,1,2\}$ and $M>0.$ The class of admissible function $\Phi_{S,1}[\Omega,M]$ consists of those functions $\phi:\mathbb{C}^{4}\times \mathcal{U}\longrightarrow\mathbb{C}$ such that
\begin{multline}\label{z42}
\phi \bigg(1+Me^{i\theta},\frac{a+[k+a]Me^{i\theta}}{a},\frac{L+a(a-1)+[a(2k+a-1)]Me^{i\theta}}{a(a-1)},\\ \frac{N+3aL+a(a-1)(a-2)+[(a-1)a(3k+a-2)]Me^{i\theta}}{a(a-1)(a-2)};z\bigg)\not\in \Omega.
\end{multline}
whenever $z \in \mathcal{U},\Re(Le^{-i\theta})\geq (k-1)kM,$ and $\Re(Ne^{-i\theta})\geq 0$ for all $\theta \in \mathbb{R}$ and $k \geq 2.$ 
\end{defn}
\begin{corollary}\label{c5}
Let $\phi \in \Phi_{S,1}[\Omega,M].$ If the function $f \in \mathcal{A}$ satisfies 
\begin{equation*}
\left|\frac{S_{a,c}f(z)}{z}\right|\leq kM\qquad(z \in \mathcal{U},k\geq 2; M>0),
\end{equation*}
and 
\begin{equation*}
\phi \left(\frac{S_{a+1,c}f(z)}{z},\frac{S_{a,c}f(z)}{z},\frac{S_{a-1,c}f(z)}{z},\frac{S_{a-2,c}f(z)}{z};z \right)\in \Omega,
\end{equation*}
then 
\begin{equation*}
\left|\frac{S_{a+1,c}f(z)}{z}-1\right|<M.
\end{equation*}
\end{corollary}
 In this special case $\Omega=q(\mathcal{U})=\{w:|w-1|<M\},$ the class $\Phi_{S,1}[\Omega,M]$ is simply denoted by $\Phi_{S,1}[M].$ Corollary \ref{c5}  can now be written in the following form
 \begin{corollary}\label{c6}
 Let $\phi \in \Phi_{S,1}[M].$ If the function $f \in \mathcal{A}$ satisfies 
\begin{equation*}
\left|\frac{S_{a,c}f(z)}{z}\right|\leq kM \qquad(k\geqslant 2 ; M>0),
\end{equation*}
and 
\begin{equation*}
\left|\phi\left(\frac{S_{a+1,c}f(z)}{z},\frac{S_{a,c}f(z)}{z},\frac{S_{a-1,c}f(z)}{z},\frac{S_{a-2,c}f(z)}{z};z\right) -1\right|<M,
\end{equation*}
then
\begin{equation*}
\left|\frac{S_{a+1,c}f(z)}{z}-1\right|< M.
\end{equation*}
\end{corollary}
\begin{corollary}\label{c20}
Let $\Re(a) \geq \frac{-k}{2},$ $0 \neq a \in \mathbb{C}, k \geq 2$ and $M>0.$ If $f \in \mathcal{A}$ satisfies
\begin{equation*}
\left|\frac{S_{a,c}f(z)}{z}\right| \leq kM,
~\text{
and}~
\left|\frac{S_{a,c}f(z)}{z}-1\right| < M,
\end{equation*}
then
\begin{equation*}
\left|\frac{S_{a+1,c}f(z)}{z}-1\right| < M.
\end{equation*}
\begin{proof}
This follows from Corollary \ref{c5}  by taking $\phi(\alpha,\beta,\gamma,\delta;z)=\beta-1.$ 
\end{proof}
\end{corollary}
\begin{remark}\label{ab}
For $f(z)=\frac{z}{1-z}$ in Corollary \ref{c20}, we have 
\begin{equation}\label{e345}
\left|\frac{U_{a,c}(z)}{z}-1\right|<M\Longrightarrow \left|\frac{U_{a+1,c}(z)}{z}-1\right|<M,
\end{equation}
which is  given by Andras and Baricz \cite{andras}.
\end{remark}
\begin{example} For $p=\pm 1/2, ~b=1$ and $c=1$, we have 
	$
	U_{2,1}(z)=zU_{1/2,1,1}(z)=2(1-\cos(\sqrt{z}))$ \quad and \quad $U_{1,1}(z)=zU_{-1/2,1,1}(z)=\sqrt{z}\sin\sqrt{z},
	$
	where $U_{p,b,c}$ is given by (\ref{z9}). Therefore, from Remark \ref{ab}, we get
	\begin{equation*}
		\left|\frac{\sqrt{z}\sin\sqrt{z}}{z}-1\right|<M\Rightarrow \left|\frac{1-\cos\sqrt{z}}{z}-\frac{1}{2}\right|<\frac{M}{2}.
	\end{equation*}
\end{example}
\begin{corollary}\label{c21}
Let $a \in \mathbb{C}\setminus \{0,1,-2\},k \geq 2$ and $M>0.$ If $f \in \mathcal{A}$ satisfies 
\begin{equation*}
\left|\frac{S_{a,c}f(z)}{z}\right| \leq kM,
~\text{
and }~
\left|\frac{S_{a-1,c}f(z)}{z}-\frac{S_{a,c}f(z)}{z}\right|< \frac{2(a+2)M}{a(a-1)},
\end{equation*}
then
\begin{equation*}
\left|\frac{S_{a+1,c}f(z)}{z}-1\right|<M.
\end{equation*}
\begin{proof}
This follows from Corollary \ref{c5} by taking $\phi(\alpha,\beta,\gamma,\delta;z)=\gamma-\beta.$
\end{proof}
\end{corollary}
\begin{defn}\label{d5a}
Let $\Omega$  be a set in $\mathbb{C}$ and $q \in \mathcal{Q}_{1} \cap \mathcal{H}_{1}.$ The class of admissible functions $\Phi_{S,2}[\Omega,q]$ consists of those functions $\phi: \mathbb{C}^{4} \times\mathcal U \longrightarrow \mathbb{C}$ that satisfy the following admissibility conditions
\begin{equation*}
\phi(v,w,x,y,;z) \not\in \Omega
\end{equation*}
whenever
\begin{equation*}
v=q(\zeta), ~w=\frac{1}{(a-1)}\left(\frac{k\zeta q^{\prime}(\zeta)}{q(\zeta)}+a q(\zeta)-1\right),
\end{equation*}
\begin{multline*}
\Re
\bigg(\frac{[(a-2)\gamma-(a-1)\beta+1](a-1)\beta\alpha}{(a-1)\beta\alpha-a\alpha^{2}+\alpha}+(a-1)\beta+1\bigg)\geq k \Re\bigg(\frac{\zeta q^{\prime\prime}(\zeta)}{q^{\prime}(\zeta)}+1\bigg),
\end{multline*}
and
\begin{multline*}
\Re \bigg[\bigg(\delta\gamma\beta\alpha(a-1)(a-2)(a-3)-(a-2)^{2}(a-1)\beta\gamma^{2}\alpha-\beta\gamma\alpha(a-1)(a-2)-\\ \qquad \qquad  \gamma\beta^{2}\alpha(a-1)^{2}(a-2)+\beta^{3}
\alpha(a-1)^{3}-2\beta\alpha(a-1)+\beta^{2}\alpha(a-1)^{2}-\alpha\beta\gamma a(a-1)(a-2)+\\\qquad \qquad \beta^{2}\alpha a (a-1)^{2}-a (a-1)\beta \alpha \bigg)\times \bigg((a-1)\beta\alpha-a\alpha^{2}+\alpha\bigg)^{-1}+ 3\gamma \beta(a-1)(a-2)-4\beta(a-1)(a\alpha-1)-\\ \qquad \qquad  2\beta^{2}(a-1)^{2}-\beta a (a-1)- 3a \alpha \beta (a-1)+2a^{2}\alpha-a+4a^{2}\alpha^{2}+a \alpha\bigg]\geq k^{2}\Re \bigg(\frac{{\zeta}^{2}q^{\prime\prime\prime}(\zeta)}{q^{\prime}(\zeta)}\bigg),
\end{multline*}
where $z \in \mathcal{U},\zeta \in \partial\mathcal{U} \setminus E(q)$ and $k\geq 2$.
\end{defn}
\begin{theorem}\label{t11}
Let $\phi \in \Phi_{S,2}[\Omega,q]$. If the function $f \in \mathcal{A}$ and $q \in \mathcal{Q}_{1}$ satisfy the following conditions
\begin{equation}\label{e30}
\Re\left(\frac{\zeta q^{\prime\prime}(\zeta)}{q^{\prime}(\zeta)}\right)\geq 0,\left| \frac{S_{a-1,c}f(z)}{S_{a,c}f(z)q^{\prime}(\zeta)}\right|\leq k,
\end{equation} 
\begin{equation}\label{e31}
\left\{\phi 
\left(\frac{S_{a,c}f(z)}{S_{a+1,c}f(z)},\frac{S_{a-1,c}f(z)}{S_{a,c}f(z)},\frac{S_{a-2,c}f(z)}{S_{a-1,c}f(z)},\frac{S_{a-3,c}f(z)}{S_{a-2,c}f(z)};z\right)
:z \in \mathcal{U}\right\}\subset \Omega,
\end{equation}
then
\begin{equation*}
\frac{S_{a,c}f(z)}{S_{a+1,c}f(z)}\prec q(z)\qquad(z \in \mathcal{U}).
\end{equation*}
\begin{proof}
Define the analytic function $p(z)$ in $\mathcal{U}$ by
\begin{equation}\label{e32}
p(z)=\frac{S_{a,c}f(z)}{S_{a+1,c}f(z)}.
\end{equation}
From equation (\ref{z11}) and (\ref{e32}), we have
\begin{equation}\label{e33}
\frac{S_{a-1,c}f(z)}{S_{a,c}f(z)}=\frac{1}{a-1}\left[\frac{zp^{\prime}(z)}{p(z)}+ap(z)-1\right]:=\frac{A}{(a-1)}.
\end{equation}
Similar computation yields:
\begin{equation}\label{e34}
\frac{S_{a-2,c}f(z)}{S_{a-1,c}f(z)}:=\frac{B}{(a-2)}
\end{equation}
and
\begin{eqnarray}\label{e35}
\frac{S_{a-3,c}f(z)}{S_{a-2,c}f(z)}=\frac{1}{a-3}\left[B-1+B^{-1}(C+A^{-1}D-A^{-2}C^{2})\right],
\end{eqnarray}
where
\begin{align*}
B&:=\frac{zp^{\prime}(z)}{p(z)}+ap(z)-2+\dfrac{\frac{zp^{\prime}(z)}{p(z)}+\frac{z^{2}p^{\prime\prime}(z)}{p(z)}-\left(\frac{zp^{\prime}(z)}{p(z)}\right)^{2}+azp^{\prime}(z)}{\frac{zp^{\prime}(z)}{p(z)}+ap(z)-1},\\
C&:=\frac{zp^{\prime}(z)}{p(z)}+\frac{z^{2}p^{\prime\prime}(z)}{p(z)}-\left(\frac{zp^{\prime}(z)}{p(z)}\right)^{2}+azp^{\prime}(z),\\
D&:= \frac{3z^{2}p^{\prime\prime}(z)}{p(z)}+\frac{zp^{\prime}(z)}{p(z)}-3\left(\frac{zp^{\prime}(z)}{p(z)}\right)^{2}+\frac{z^{3}p^{\prime\prime\prime}(z)}{p(z)}-\frac{3z^{3}p^{\prime\prime}(z)p^{\prime}(z)}{(p(z))^{2}}+\\ &\qquad 2\left(\frac{zp^{\prime}(z)}{p(z)}\right)^{3}  +azp^{\prime}+az^{2}p^{\prime\prime}(z).
\end{align*}
Define the transformation from $\mathbb{C}^{4}$ to $\mathbb{C}$ by
\begin{eqnarray*}
\alpha(r,s,t,u)=r,\qquad \beta(r,s,t,u)=\frac{1}{a-1}\left[\dfrac{s}{r}+ar-1 \right]:=\frac{E}{(a-1)},
\end{eqnarray*}
\begin{eqnarray}\label{e36}
\gamma(r,s,t,u)=\frac{1}{a-2}\left[\frac{s}{r}+ar-2+\frac{\frac{t}{r}+\frac{s}{r}-(\frac{s}{r})^{2}+as}{\frac{s}{r}+ar-1} \right]:=\frac{F}{(a-2)},
\end{eqnarray}
and
\begin{eqnarray}\label{e37}
\delta(r,s,t,u)= \frac{1}{a-3}\left[F-1+F^{-1}\left(G+E^{-1}H-E^{-2}G^{2}\right)\right],
\end{eqnarray}
where
\begin{align*}
G&:=\frac{t}{r}+\frac{s}{r}+as-\left(\frac{s}{r}\right)^{2},\\ 
H&:=\frac{3t}{r}+\frac{s}{r}-3\left(\frac{s}{r}\right)^{2}+\frac{u}{r}-\frac{3st}{r^{2}}+2\left(\frac{s}{r}\right)^{3}+ as-at.
\end{align*}
Let
\begin{multline}\label{e38}
\psi(r,s,t,u)=\phi(\alpha,\beta,\gamma,\delta;z)\\=\phi \left(r,\frac{E}{a-1},\frac{F}{a-2},\left[\frac{F-1+F^{-1}\left(G+E^{-1}H-E^{-2}G^{2}\right)}{a-3}\right]\right).
\end{multline}
The proof will make use of  Lemma \ref{t4}. Using equations (\ref{e32}) to (\ref{e35}),  and from (\ref{e38}), we have
\begin{multline}\label{e39}
\psi \left(p(z),zp^{\prime}(z),z^{2}p^{\prime\prime}(z),z^{3}p^{\prime\prime\prime}(z);z\right)\\=\phi \left(\frac{S_{a,c}f(z)}{S_{a+1,c}f(z)},\frac{S_{a-1,c}f(z)}{S_{a,c}f(z)},\frac{S_{a-2,c}f(z)}{S_{a-1,c}f(z)},\frac{S_{a-3,c}f(z)}{S_{a-2,c}f(z)};z\right).
\end{multline} 
Hence, (\ref{e31}) becomes
\begin{equation*}
\psi \left(p(z),zp^{\prime}(z),z^{2}p^{\prime\prime}(z),z^{3}p^{\prime\prime\prime}(z);z\right) \in \Omega.
\end{equation*}
Note that
\begin{equation*}
\frac{t}{s}+1=\left(\frac{[(a-2)\gamma-(a-1)\beta+1](a-1)\beta\alpha}{(a-1)\beta\alpha-a\alpha^{2}+\alpha}+(a-1)\beta+1\right)
\end{equation*}
and
\begin{multline*}
\frac{u}{s}=\bigg[\bigg(\delta\gamma\beta\alpha(a-1)(a-2)(a-3)-(a-2)^{2}(a-1)\beta\gamma^{2}\alpha-\beta\gamma\alpha(a-1)(a-2)-\\ \qquad \qquad \gamma\beta^{2}\alpha(a-1)^{2}(a-2)+\beta^{3}\alpha(a-1)^{3}-2\beta\alpha(a-1)+\beta^{2}\alpha(a-1)^{2}-\alpha\beta\gamma a(a-1)(a-2)+\\ \qquad \qquad \beta^{2}\alpha a (a-1)^{2}-a (a-1)\beta \alpha \bigg)\times \bigg((a-1)\beta\alpha-a\alpha^{2}+\alpha\bigg)^{-1}+ 3\gamma \beta(a-1)(a-2)-4\beta(a-1)\\ \qquad \qquad (a\alpha-1)-2\beta^{2}(a-1)^{2}-\beta a (a-1)- 3a \alpha \beta (a-1)+2a^{2}\alpha-a+4a^{2}\alpha^{2}+a \alpha\bigg].
\end{multline*}
Thus, the admissibility condition for $\phi \in \Phi_{S,2}[\Omega,q]$ in Definition \ref{d5} is equivalent to the admissibility condition for $\psi \in \Psi_{2}[\Omega,q]$ as given in Definition \ref{d2a} with $n=2.$ Therefore, by using (\ref{e30}) and Lemma \ref{t4}, we have
\begin{equation*}
\frac{S_{a,c}f(z)}{S_{a+1,c}f(z)}\prec q(z).
\end{equation*}
This completes the proof of theorem.
\end{proof}
\end{theorem}
If $\Omega \neq \mathbb{C}$ is a simply connected domain, then $\Omega=h(\mathcal{U})$ for some conformal mapping $h(z)$ of $\mathcal{U}$ onto $\Omega.$ In this case, the class $\Phi_{S,1}[h(\mathcal{U}),q]$ is written  as $\Phi_{S,2}[h,q]$.
Proceeding similarly as in the previous theorem, the following result is an immediate consequence of Theorem   \ref{t11}.
\begin{theorem}\label{t12}
Let $\phi \in \Phi_{S,2}[h,q].$If the function $f \in \mathcal{A}$ and $q \in \mathcal{Q}_{1}$ satisfy the following conditions $(\ref{e30})$ and 
\begin{equation}\label{e40}
\phi \left(\frac{S_{a,c}f(z)}{S_{a+1,c}f(z)},\frac{S_{a-1,c}f(z)}{S_{a,c}f(z)},\frac{S_{a-2,c}f(z)}{S_{a-1,c}f(z)},\frac{S_{a-3,c}f(z)}{S_{a-2,c}f(z)};z\right)\prec h(z),
\end{equation}
then 
\begin{equation*}
\frac{S_{a,c}f(z)}{S_{a+1,c}f(z)} \prec q(z)\qquad(z \in \mathcal{U}).
\end{equation*}
\end{theorem}
 \section{Results based on differential superordination}
 In this section, the third-order differential superordination theorems for the operator $S_{a,c}$  defined in (\ref{z10}) is investigated. For the purpose, we considered the following class of admissible functions.
 \begin{defn}\label{d7}
 Let $\Omega$ be a set in $\mathbb{C}$ and $q \in \mathcal{H}_{0}$ with $q^{\prime}(z) \neq 0.$ The class of admissible function $\Phi_{S}^{\prime}[\Omega,q]$ consists of those functions $\phi:\mathbb{C}^{4}\times \overline{\mathcal{U}}\longrightarrow\mathbb{C}$ that satisfy the following admissibility conditions:
\begin{equation*}
\phi(\alpha,\beta,\gamma,\delta;\zeta) \in \Omega
\end{equation*}
whenever
\begin{equation*}
\alpha=q(z),\beta=\frac{z q^{\prime}(z)+m(a-1)q(z)}{ma},
\end{equation*}
\begin{equation*}
\Re\left(\frac{a(a-1)\gamma-(a-2)(a-1)\alpha}{a\beta-(a-1)\alpha}-(2a-3)\right)\leq \frac{1}{m} \Re\left(\frac{z q^{\prime\prime}(z)}{q^{\prime}(z)}+1\right),
\end{equation*}
and 
\begin{equation*}
\Re\left(\frac{a(a-1)((1-a)\alpha+(3a\beta+(1-3a)\gamma+(a-2)\delta)}{\alpha+a(\beta-\alpha)}\right)\leq \frac{1}{m^{2}} \Re \left(\frac{z^{2} q^{\prime\prime\prime}(z)}{q^{\prime}(z)}\right),
\end{equation*}
where $z \in \mathcal{U},\zeta \in \partial \mathcal{U} \setminus E(q),$ and $m\geq 2.$ \end{defn}
\begin{theorem}\label{t13}
Let $\phi \in \Phi_{S}^{\prime}[\Omega,q].$If the function $f \in \mathcal{A}$ and $S_{a+1,c}f(z)\in \mathcal{Q}_{0}$ and $q \in \mathcal{H}_{0}$ with $q^{\prime}(z)\neq 0$ satisfy the following conditions:
\begin{equation}\label{p30}
\Re\left(\frac{z q^{\prime\prime}(z)}{q^{\prime}(z)}\right)\geq 0,\left|\frac{S_{a,c}f(z)}{q^{\prime}(z)}\right|\leq m,
\end{equation}
and 
\begin{equation*}
\phi(S_{a+1,c}f(z),S_{a,c}f(z),S_{a-1,c}f(z),S_{a-2,c}f(z);z)
\end{equation*}
 is univalent in $\mathcal{U}$,then 
 \begin{equation}\label{p31}
\Omega \subset \{\phi(S_{a+1,c}f(z),S_{a,c}f(z),S_{a-1,c}f(z),S_{a-2,c}f(z);z)z \in \mathcal{U}\},
 \end{equation}
 implies that
 \begin{equation*}
q(z) \prec S_{a+1,c}f(z)\qquad(z \in \mathcal{U}).
\end{equation*}
\end{theorem}
\begin{proof}
Let the function $p(z)$ be defined by (\ref{z18}) and $\psi$ by  (\ref{z24}). Since $\phi \in \Phi_{S}^{\prime}[\Omega,q].$ From (\ref{z25}) and (\ref{p31}) yield
\begin{equation*}
\Omega \subset \{\psi \left(p(z),zp^{\prime}(z),z^{2}p^{\prime\prime}(z),z^{3}p^{\prime\prime\prime}(z);z\right)z \in \mathcal{U}\}.
\end{equation*}
From (\ref{z22}) and (\ref{z23}), we see that  the admissibility condition for $\phi \in \Phi_{S}^{\prime}[\Omega,q]$ in Definition \ref{d7} is equivalent to the admissibility condition for $\psi \in \Psi_{2}[\Omega,q]$ as given in Definition \ref{d6}  with $n=2.$ Hence $\psi \in \Psi_{2}^{\prime}[\Omega,q]$  and by using (\ref{p31}) and Lemma \ref{t5}, we have
\begin{equation*}
q(z) \prec S_{a+1,c}f(z).
\end{equation*}
This completes the proof of theorem.
\end{proof}
If $\Omega \neq \mathbb{C}$ is a simply connected domain, then $\Omega=h(\mathcal{U})$ for some conformal mapping $h(z)$ of $\mathcal{U}$ onto $\Omega.$ In this case, the class $\Phi_{S}^{\prime}[h(\mathcal{U}),q]$ is written  as $\Phi_{S}^{\prime}[h,q]$. The following result is an immediate  consequence of Theorem  \ref{t13}.
\begin{theorem}\label{t14}
Let $\phi \in \Phi_{S}^{\prime}[h,q]$ and $h$ be analytic in $\mathcal{U}$. If the function $f \in \mathcal{A},~S_{a+1,c}f(z)\in\mathcal Q_0$ and  $q \in  \mathcal{H}_{0}$  with $q ^{\prime}(z)\neq 0$ satisfy the following conditions (\ref{p30}) and 
\begin{equation*}
\phi(S_{a+1,c}f(z),S_{a,c}f(z),S_{a-1,c}f(z),S_{a-2,c}f(z)
;z),
\end{equation*}
is univalent in $\mathcal{U},$ then
\begin{equation}\label{p32}
h(z) \prec\phi(S_{a+1,c}f(z),S_{a,c}f(z),S_{a-1,c}f(z),S_{a-2,c}f(z)
;z)
\end{equation}
implies that 
\begin{equation*}
q(z) \prec S_{a+1,c}f(z)\qquad(z \in \mathcal{U}).
\end{equation*}
\end{theorem}
Theorem \ref{t13} and \ref{t14} can only be used to obtain subordinants of the third-order differential superordination of the forms (\ref{p31}) or (\ref{p32}). The next result shows the existence of the best subordinant of (\ref{p32}) for a suitable $\phi.$
\begin{theorem}\label{t15}
Let the function $h$ be univalent in $\mathcal{U}$ and let $\phi : \mathbb{C}^{4}\times \overline{\mathcal{U}}\longrightarrow \mathbb{C}$ and $\psi$ be given by (\ref{z24}). Suppose that the differential equation
\begin{equation}\label{p33}
\psi(q(z),zq^{\prime}(z),z^{2}q^{\prime\prime}(z),z^{3}
q^{\prime\prime\prime}(z);z)=h(z),
\end{equation}
has a solution $q(z) \in \mathcal{Q}_{0}.$ If the function $f \in \mathcal{A},~S_{a+1,c}f(z)\in\mathcal Q_0$ and $q \in \mathcal{H}_{0}$ with $q^{\prime}(z) \neq 0,$ which satisfy  the following condition (\ref{p30}) and 
\begin{equation*}
\phi(S_{a+1,c}f(z),S_{a,c}f(z),S_{a-1,c}f(z),S_{a-2,c}f(z));z)
\end{equation*}
is analytic in $\mathcal{U},$ then
\begin{equation*}
h(z) \prec \phi(S_{a+1,c}f(z),S_{a,c}f(z),S_{a-1,c}f(z),S_{a-2,c}f(z);z)
\end{equation*}
implies that
\begin{equation*}
q(z) \prec S_{a+1,c}f(z)\qquad(z \in \mathcal{U}).
\end{equation*}
and $q(z)$ is the best dominant.
\begin{proof}
In view of Theorem \ref{t13} and Theorem \ref{t14} we deduce that $q$ is a subordinant of (\ref{p32}). Since $q$ satisfies (\ref{p33}), it is also a solution of (\ref{p32}) and therefore $q$ will be subordinated by all subordinants. Hence $q$ is the best subordinant. This completes the proof of theorem.
\end{proof}
\end{theorem}
\begin{defn}\label{d8}
Let $\Omega$ be a set in $\mathbb{C}, q \in \mathcal{H}_{0}$ with $q'(z)\neq 0$.  The class of admissible functions $\Phi_{S,1}^{\prime}[\Omega,q]$ consists of those functions $\phi:\mathbb{C}^{4}\times \overline{\mathcal{U}}\longrightarrow \mathbb{C}$ that satisfy the following admissibility condition
\begin{equation*}
 \phi(\alpha,\beta,\gamma,\delta;z) \in \Omega,
 \end{equation*}
whenever
\begin{equation*}
\alpha=q(z),\beta=\frac{z q^{\prime}(z)+amq(z)}{am}, 
\end{equation*}
\begin{equation*}
\Re\left(\frac{(a-1)(\gamma-\alpha)}{\beta-\alpha}+(1-2a)\right)\leq \frac{1}{m} \Re\left(\frac{z q^{\prime\prime}(z)}{q^{\prime}(z)}+1\right),
\end{equation*}
and 
\begin{eqnarray*}
&&\Re \left(\frac{(a-1)(a-2)(\delta-\alpha)-3a(a-1)(\gamma-2\alpha+\beta)}{\beta-\alpha}+6a^{2}\right)\leq \frac{1}{m^{2}} \Re \left(\frac{z^{2} q^{\prime\prime\prime}(z)}{q^{\prime}(z)}\right),
\end{eqnarray*}
where $z \in \mathcal{U},\zeta \in \partial \mathcal{U} \setminus E(q),$ and $m\geq 2.$
\end{defn}
\begin{theorem}\label{t16}
Let $\phi \in \Phi_{S,1}^{\prime}[\Omega,q].$ If the function $f \in \mathcal{A}, \frac{S_{a+1,c}f(z)}{z} \in \mathcal{Q}_{0}$ and $q \in \mathcal{H}_{0}$ with $q^{\prime}(z) \neq 0$ satisfy the following conditions:
\begin{equation}\label{p34}
\Re \left(\frac{z q^{\prime\prime}(z)}{q^{\prime}(z)}\right)\geq 0,\quad \left|\frac{S_{a,c}f(z)}{zq^{\prime}(z)}\right|\leq m
\end{equation}
and 
\begin{equation*}
\phi \left(\frac{{S}_{a+1,c}f(z)}{z},\frac{S_{a,c}f(z)}{z},\frac{S_{a-1,c}f(z)}{z},\frac{S_{a-2,c}f(z)}{z};z \right),
\end{equation*}
is univalent  in $\mathcal{U},$
then 
\begin{equation}\label{p35}
\Omega \subset \left\{\phi \left(\frac{{S}_{a+1,c}f(z)}{z},\frac{S_{a,c}f(z)}{z},\frac{S_{a-1,c}f(z)}{z},\frac{S_{a-2,c}f(z)}{z};z \right): z \in \mathcal{U} \right\},
\end{equation}
implies that
\begin{equation*}
q(z) \prec \frac{S_{a+1,c}f(z)}{z}\qquad(z \in \mathcal{U}).
\end{equation*}
\begin{proof}
Let the function $p(z)$ be defined by (\ref{z32}) and $\psi$ by (\ref{z38}). Since $\phi \in \Phi _{S,1}^{\prime}[\Omega,q],$ (\ref{z39}) and (\ref{p35}) yield
\begin{equation*}
\Omega \subset \left\{\psi \left(p(z),zp^{\prime}(z),z^{2}p^{\prime\prime}(z),z^{3}p^{\prime\prime\prime}(z);z\right):z \in \mathcal{U}\right\}.
\end{equation*}
From equations (\ref{z36}) and (\ref{z37}), we see that the admissible condition for $\phi \in \Phi_{S,1}^{\prime}[\Omega,q]$ in Definition \ref{d8} is equivalent to the admissible condition for $\psi$ as given in Definition \ref{d6} with $n=2.$ Hence $\psi \in \Psi _{2}^{\prime}[\Omega,q]$ and by using (\ref{p34}) and Lemma \ref{t5}, we have
\begin{equation*}
q(z)\prec \frac{S_{a+1,c}f(z)}{z}.
\end{equation*} 
This completes the proof of theorem.
\end{proof}
\end{theorem}
If $\Omega \neq \mathbb{C}$ is a simply connected domain, then $\Omega=h(\mathcal{U})$ for some conformal mapping $h(z)$ of $\mathcal{U}$ onto $\Omega.$ In this case, the class $\Phi_{S,1}^{\prime}[h(\mathcal{U}),q]$ is written  as $\Phi_{S,1}^{\prime}[h,q]$, The following result is an immediate  consequence of Theorem  \ref{t16}.
\begin{theorem}\label{t17}
Let $\phi \in \Phi_{S,1}^{\prime}[h,q].$ If the function $f \in \mathcal{A}, \frac{S_{a+1,c}f(z)}{z} \in \mathcal{Q}_{0}$ and $q \in \mathcal{H}_{0}$ with $q^{\prime}(z)\neq 0$ satisfy the following conditions:
\begin{equation*}
\Re \left(\frac{z q^{\prime\prime}(z)}{q^{\prime}(z)}\right)\geq 0,\left|\frac{S_{a,c}f(z)}{zq^{\prime}(z)}\right|\leq m,
\end{equation*}
and 
\begin{equation*}
\phi \left(\frac{{S}_{a+1,c}f(z)}{z},\frac{S_{a,c}f(z)}{z},\frac{S_{a-1,c}f(z)}{z},\frac{S_{a-2,c}f(z)}{z};z \right)
\end{equation*}
is univalent in $\mathcal{U},$
then
 \begin{equation*}
h(z)\prec\phi \left(\frac{{S}_{a+1,c}f(z)}{z},\frac{S_{a,c}f(z)}{z},\frac{S_{a-1,c}f(z)}{z},\frac{S_{a-2,c}f(z)}{z};z \right),
\end{equation*}
implies that
\begin{equation*}
q(z) \prec\frac{S_{a+1,c}f(z)}{z}\qquad(z \in \mathcal{U}).
\end{equation*}
\end{theorem}

\begin{defn}\label{d9}
Let $\Omega$  be a set in $\mathbb{C}$ and $q \in \mathcal{H}_{1}$ with $q^{\prime}(z)\neq 0.$ The class of admissible functions $\Phi_{S,2}^{\prime}[\Omega,q]$ consists of those functions $\phi: \mathbb{C}^{4} \times \overline{\mathcal{U}}\longrightarrow \mathbb{C}$ that satisfy the following admissibility conditions
\begin{equation*}
\phi(\alpha,\beta,\gamma,\delta;z) \in \Omega,
\end{equation*}
whenever
\begin{equation*}
\alpha=q(z), \beta=\frac{1}{a-1}\left(\frac{z q^{\prime}(z)}{mq(z)}+a q(z)-1\right),
\end{equation*}
\begin{equation*}
\Re\left(\frac{[(a-2)\gamma-(a-1)\beta+1](a-1)\beta\alpha}{(a-1)\beta\alpha-a\alpha^{2}+\alpha}+(a-1)\beta+1\right)\leq \frac{1}{m}\Re\left(\frac{z q^{\prime\prime}(z)}{q^{\prime}(z)}+1\right),
\end{equation*}and
\begin{multline*}
\Re \bigg[\bigg(\delta\gamma\beta\alpha(a-1)(a-2)(a-3)-(a-2)^{2}(a-1)\beta\gamma^{2}\alpha-\beta\gamma\alpha(a-1)(a-2)-\gamma\beta^{2}\alpha(a-1)^{2}(a-2)+\\ \qquad\qquad \beta^{3}\alpha(a-1)^{3}-2\beta\alpha(a-1)+ \beta^{2}\alpha(a-1)^{2}-\alpha\beta\gamma a(a-1)(a-2)+\beta^{2}\alpha a (a-1)^{2}-a (a-1)\beta \alpha \bigg)\times \\ \qquad\qquad \bigg((a-1)\beta\alpha-a\alpha^{2}+\alpha\bigg)^{-1}+ 3\gamma \beta(a-1)(a-2)-4\beta(a-1)(a\alpha-1)-2\beta^{2}(a-1)^{2}-\beta a (a-1)-\\  3a \alpha \beta (a-1)+2a^{2}\alpha-a+4a^{2}\alpha^{2}+a \alpha\bigg]\leq \frac{1}{m^{2}}\Re\bigg(\frac{{z}^{2}q^{\prime\prime\prime}(z)}{q^{\prime}(z)}\bigg),
\end{multline*}
where $z \in \mathcal{U},\zeta \in \partial \mathcal{U} \setminus E(q)$ and $m\geq 2.$
\end{defn}
\begin{theorem}\label{t18}
Let $\phi \in \Phi_{S,2}^{\prime}[\Omega,q].$ If the function $f \in \mathcal{A}$ and $\frac{S_{a,c}f(z)}{S_{a+1,c}f(z)} \in \mathcal{Q}_{1}$ and $q \in \mathcal{H}_{1}$ with $q^{\prime}(z)\neq 0$ satisfy the following conditions
\begin{equation}\label{p36}
\Re\left(\frac{z q^{\prime\prime}(z)}{q^{\prime}(z)}\right)\geq 0,\left| \frac{S_{a-1,c}f(z)}{S_{a,c}q^{\prime}(z)}\right|\leq m,
\end{equation} 
and
\begin{equation*}
\phi \left(\frac{S_{a,c}f(z)}{S_{a+1,c}f(z)},\frac{S_{a-1,c}f(z)}{S_{a,c}f(z)},\frac{S_{a-2,c}f(z)}{S_{a-1,c}f(z)},\frac{S_{a-3,c}f(z)}{S_{a-2,c}f(z)};z\right)
\end{equation*}
is univalent in $\mathcal{U},$ then
\begin{equation}\label{p37}
\Omega \subset\left\{\phi \left(\frac{S_{a,c}f(z)}{S_{a+1,c}f(z)},\frac{S_{a-1,c}f(z)}{S_{a,c}f(z)},\frac{S_{a-2,c}f(z)}{S_{a-1,c}f(z)},\frac{S_{a-3,c}f(z)}{S_{a-2,c}f(z)}
;z\right):z \in \mathcal{U}\right\},
\end{equation}
then
\begin{equation*}
 q(z) \prec \frac{S_{a,c}f(z)}{S_{a+1,c}f(z)}\qquad(z \in \mathcal{U}).
\end{equation*}
\begin{proof}
Let the function $p(z)$ be defined by (\ref{e32}) and $\psi$ by  (\ref{e38}). Since $\phi \in \Phi_{S,2}^{\prime}[\Omega,q],$ (\ref{e39}) and (\ref{p37}) yield 
\begin{equation*}
\Omega \subset \{\psi \left(p(z),zp^{\prime}(z),z^{2}p^{\prime\prime}(z),z^{3}p^{\prime\prime\prime}(z);z\right)z \in \mathcal{U}\}.
\end{equation*}
From equations (\ref{e36}) and (\ref{e37}), we see that the admissible condition for $\phi \in \Phi_{S,2}^{\prime}[\Omega,q]$ in Definition \ref{d9} is equivalent to the admissible condition for $\psi$ as given in Definition \ref{d6} with $ n=2.$ Hence $\psi \in \Psi_{2}^{\prime}[\Omega,q],$ and by using (\ref{p36}) and Lemma \ref{t5}, we have 
\begin{equation*}
q(z)\prec \frac{S_{a,c}f(z)}{S_{a+1,c}f(z)}\qquad(z \in \mathcal{U}).
\end{equation*}
This completes the proof of theorem.
\end{proof}
\end{theorem}
If $\Omega \neq \mathbb{C}$ is a simply connected domain, then $\Omega=h(\mathcal{U}),$ for some conformal mapping h(z) of $\mathcal{U}$ on to $\Omega$. In this case  the class $\Phi_{S,2}^{\prime}[h(\mathcal{U}),q]$
is written as $\Phi_{S}^{\prime}[h,q]$.  The following result is a consequence of Theorem \ref{t18}.
\begin{theorem}\label{t18a}
Let $\phi \in \Phi_{S,2}^{\prime}[\Omega,q].$ If the function $f \in \mathcal{A}$ and $\frac{S_{a,c}f(z)}{S_{a+1,c}f(z)} \in \mathcal{Q}_{1}$ and $q \in \mathcal{H}_{1}$ with $q^{\prime}(z)\neq 0$ satisfy the following conditions
\begin{equation}\label{p36}
\Re\left(\frac{z q^{\prime\prime}(z)}{q^{\prime}(z)}\right)\geq 0,\left| \frac{S_{a-1,c}f(z)}{S_{a,c}q^{\prime}(z)}\right|\leq m,
\end{equation} 
and
\begin{equation*}
\phi \left(\frac{S_{a,c}f(z)}{S_{a+1,c}f(z)},\frac{S_{a-1,c}f(z)}{S_{a,c}f(z)},\frac{S_{a-2,c}f(z)}{S_{a-1,c}f(z)},\frac{S_{a-3,c}f(z)}{S_{a-2,c}f(z)};z\right)
\end{equation*}
is univalent in $\mathcal{U},$ then
\begin{equation}\label{p37}
h(z)\prec \phi \left(\frac{S_{a,c}f(z)}{S_{a+1,c}f(z)},\frac{S_{a-1,c}f(z)}{S_{a,c}f(z)},\frac{S_{a-2,c}f(z)}{S_{a-1,c}f(z)},\frac{S_{a-3,c}f(z)}{S_{a-2,c}f(z)}
;z\right)
\end{equation}
then
\begin{equation*}
 q(z) \prec \frac{S_{a,c}f(z)}{S_{a+1,c}f(z)}\qquad(z \in \mathcal{U}).
\end{equation*}
\end{theorem}
\section{Sandwich type results}
Combining Theorem \ref{t7} and \ref{t14}, we obtain the following sandwich-type theorem.
\begin{corollary}\label{c9}
Let $h_{1}$ and $q_{1}$ be analytic functions in $\mathcal{U},$ $h_{2}$ be univalent function in $\mathcal{U},$ $q_{2} \in \mathcal{Q}_{0}$ with $q_{1}(0)=q_{2}(0)=0$ and $\phi \in \Phi_{S}[h_{2},q_{2}]\cap \Phi_{S}^{\prime}[h_{1},q_{1}].$ If the function $f \in \mathcal{A}, S_{a+1,c} \in \mathcal{Q}_{0}\cap \mathcal{H}_{0},$ and  
\begin{equation*}
\phi(S_{a+1,c}f(z),S_{a,c}f(z),S_{a-1,c}f(z),S_{a-2,c}f(z);z),
\end{equation*}
is univalent in $\mathcal{U},$ and the condition (\ref{z16}) and (\ref{p30}) are satisfied, then
\begin{equation*}
h_{1}(z)\prec \phi(S_{a+1,c}f(z),S_{a,c}f(z),S_{a-1,c}f(z),S_{a-2,c}f(z);z) \prec h_{2}(z)
\end{equation*}
implies that
\begin{equation*}
q_{1}(z)\prec S_{a+1,c}f(z)\prec q_{1}(z)\qquad(z \in \mathcal{U}).
\end{equation*}
\end{corollary}
Combining Theorems \ref{t10} and \ref{t17}, we obtain the following sandwich-type theorem.
\begin{corollary}\label{c10}
Let $h_{1}$ and $q_{1}$ be analytic functions in $\mathcal{U},$ $h_{2}$ be univalent function in $\mathcal{U},$ $q_{2} \in \mathcal{Q}_{0}$ with $q_{1}(0)=q_{2}(0)=0$ and $\phi \in \Phi_{S,1}[h_{2},q_{2}]\cap \Phi_{S,1}^{\prime}[h_{1},q_{1}].$ If the function $f \in \mathcal{A},\frac{S_{a+1,c}f(z)}{z} \in \mathcal{Q}_{0}\cap \mathcal{H}_{0},$ and  
\begin{equation*}
\phi \left(\frac{{S}_{a+1,c}f(z)}{z},\frac{S_{a,c}f(z)}{z},\frac{S_{a-1,c}f(z)}{z},\frac{S_{a-2,c}f(z)}{z};z \right),
\end{equation*}
is univalent in $\mathcal{U},$ and the condition (\ref{z30}) and (\ref{p34}) are satisfied, then
\begin{equation*}
h_{1}(z)\prec \phi \left(\frac{{S}_{a+1,c}f(z)}{z},\frac{S_{a,c}f(z)}{z},\frac{S_{a-1,c}f(z)}{z},\frac{S_{a-2,c}f(z)}{z};z \right) \prec h_{2}(z)
\end{equation*}
implies that
\begin{equation*}
q_{1}(z)\prec \frac{S_{a+1,c}f(z)}{z}\prec q_{1}(z)\qquad(z \in \mathcal{U}).
\end{equation*}
\end{corollary}
Combining Theorem \ref{t12} and \ref{t18a}, we obtain the following sandwich-type theorem.
\begin{corollary}\label{c10a}
Let $h_{1}$ and $q_{1}$ be analytic functions in $\mathcal{U},$ $h_{2}$ be univalent function in $\mathcal{U},$ $q_{2} \in \mathcal{Q}_{1}$ with $q_{1}(0)=q_{2}(0)=1$ and $\phi \in \Phi_{S,1}[h_{2},q_{2}]\cap \Phi_{S,1}^{\prime}[h_{1},q_{1}].$ If the function $f \in \mathcal{A},\frac{S_{a,c}f(z)}{S_{a+1,c}f(z)} \in \mathcal{Q}_{1}\cap \mathcal{H}_{1},$ and  
\begin{equation*}
\phi \left(\frac{S_{a,c}f(z)}{S_{a+1,c}f(z)},\frac{S_{a-1,c}f(z)}{S_{a,c}f(z)},\frac{S_{a-2,c}f(z)}{S_{a-1,c}f(z)},\frac{S_{a-3,c}f(z)}{S_{a-2,c}f(z)};z\right)
\end{equation*}
is univalent in $\mathcal{U},$ and the condition (\ref{e30}) and (\ref{p37}) are satisfied, then
\begin{equation*}
h_{1}(z)\prec \phi \left(\frac{S_{a,c}f(z)}{S_{a+1,c}f(z)},\frac{S_{a-1,c}f(z)}{S_{a,c}f(z)},\frac{S_{a-2,c}f(z)}{S_{a-1,c}f(z)},\frac{S_{a-3,c}f(z)}{S_{a-2,c}f(z)};z\right) \prec h_{2}(z)
\end{equation*}
implies that
\begin{equation*}
q_{1}(z)\prec \frac{S_{a,c}f(z)}{S_{a+1,c}f(z)}\prec q_{1}(z)\qquad(z \in \mathcal{U}).
\end{equation*}
\end{corollary}
 \begin{remark}\label{r1110}
For special cases all of above results, we can obtain the corresponding results for the operators  $\mathfrak{S},$ $\mathfrak{I}$, which are defined by (\ref{z12}), and (\ref{z13}), respectively.
\end{remark}

\end{document}